\documentclass[12pt]{article}

\usepackage[T1]{fontenc}
\usepackage[utf8]{inputenc}
\usepackage{lmodern}
\usepackage{microtype}
\usepackage[margin=1in]{geometry}
\usepackage{amsmath, amssymb, amsthm, mathtools}
\usepackage{enumitem}
\usepackage[hidelinks]{hyperref}
\usepackage[capitalize,nameinlink]{cleveref}

\newtheorem{theorem}{Theorem}[section]
\newtheorem{lemma}[theorem]{Lemma}
\newtheorem{proposition}[theorem]{Proposition}

\theoremstyle{definition}
\newtheorem{definition}[theorem]{Definition}
\newtheorem{example}[theorem]{Example}
\theoremstyle{remark}
\newtheorem{remark}[theorem]{Remark}
\newtheorem{problem}[theorem]{Problem}

\title{Spanning Factorizations in Vertex-Transitive Digraphs of Degree~2}
\author{Vance Faber\\Hoquiam, WA}
\date{}

\begin{document}

\maketitle

\begin{abstract}
We investigate the existence of spanning 1-factorizations in vertex-transitive digraphs of out-degree $d$. The open question is whether every such digraph
admits a spanning 1-factorization that includes, for each vertex $v$, all $d$
out-edges $(v,F_i(v))$ from $v$. 

This paper focuses on the case $d=2$. Using the structure of alternating cycles
and block systems, we develop a block/phase framework that yields sufficient
conditions for including both $F_1,F_2$. We show that certain block obstructions
can prevent their simultaneous inclusion, while sharply transitive sets (and
hence spanning 1-factorizations) always exist. Our results provide general
constraints on feasible block sizes, describe the role of phase distributions,
and illustrate the theory with concrete families, including coset digraphs on
$A_5$. The necessity of the block criterion remains open, even in degree~2.
\end{abstract}

\section{Introduction}
The central problem that this paper addresses is to understand spanning 1-factorizations in vertex-transitive digraphs of out-degree $d$. A spanning 1-factorization is a
1-factorization together with a set of words $W$ in the factors such that, from
any root vertex $v_0$, the evaluation $\{w(v_0):w\in W\}$ covers the entire
vertex set without collision. If $W$ contains the empty word and all $d$
generators $F_1,\dots,F_d$, then every out-edge is incorporated at every vertex,
making the factorization particularly natural for routing and algebraic
purposes.

Our main results concern the case $d=2$, where the problem already exhibits rich
structure and subtle obstructions. The paper is organized as follows:
\begin{enumerate}[label=(\arabic*)]
  \item \textbf{Preliminaries.} We recall the Cayley--coset representation of
  vertex-transitive digraphs (Sabidussi’s theorem), introduce covering groups
  generated by 1-factors, and review alternating cycles.

  \item \textbf{Block/phase framework.} We construct position blocks from the
  alternating cycle structure, define tied refinements, and analyze atoms via
  phase counts. This leads to a classification of $C$-invariant refinements and
  a block-derangement criterion for including $F_1,F_2$.

  \item \textbf{Existence results.} We prove that sharply transitive sets always
  exist, guaranteeing spanning factorizations in all cases. The block criterion
  provides a sufficient condition for including both $F_1,F_2$.

  \item \textbf{Examples.} We test the framework on explicit families, including
  a simple toy model and coset digraphs of $A_5$. These illustrate both the
  obstruction and the existence of partial relocatable trees.

  \item \textbf{Open problems.} The necessity of the block criterion is still
  unresolved, even in degree~2. For general degree $d$, the overarching
  question is whether every vertex-transitive digraph admits a spanning
  1-factorization containing $\{\emptyset,F_1,\dots,F_d\}$.
\end{enumerate}

\section{Background and Definitions}

A digraph $D=(V,E)$ is \emph{regular of degree $d$} if every vertex has in-degree 
and out-degree equal to $d$. Throughout, we will restrict attention to strongly connected digraphs - 
there is a directed path between any two vertices. 

\begin{definition}[1-factors and factorizations]
A \emph{factor} of a digraph $D$ is a function $F:V\to V$ such that $(v,F(v))\in E$ 
for all $v\in V$. Equivalently, a factor may be identified with a spanning subgraph 
in which each vertex has out-degree one. A factor is a \emph{1-factor} if it is 
one-to-one, so that each vertex also has in-degree one. A \emph{factorization} of 
a regular digraph $D$ of degree $d$ is a disjoint partition of its edge set into 
$d$ factors $F_1,\dots,F_d$. If each factor is a 1-factor, the decomposition is a 
\emph{1-factorization}. By Petersen’s theorem \cite{Petersen1891}, every regular digraph admits a 
1-factorization.
\end{definition}

Each 1-factor $F_i$ is a disjoint union of directed cycles covering $V$. If 
\[
w = F_{i_k} F_{i_{k-1}} \cdots F_{i_1},
\]
 is a word in the alphabet $\{F_1,\dots,F_d\}$, 
then for a vertex $v$ the action $w(v)$ denotes the endpoint of the path obtained 
by successively applying these factors to $v$ from right to left. 

\begin{definition}[Spanning 1-factorization]
A \emph{spanning 1-factorization \cite{DoughertyFaber2020}} is a 1-factorization 
$\{F_1,\dots,F_d\}$ together with a set of words $W$ in $F=\{F_1,\dots,F_d\}$ containing 
the empty word such that for any $v_0\in V$, the set 
\[
\{\, w(v_0) : w\in W \,\}
\]
equals $V$, and each vertex of $V$ is reached by a unique word in $W$.
\end{definition}

This definition encodes a system of directed paths from any vertex to all other vertices. In network routing, 
such a spanning factorization corresponds to an “all-to-all” communication scheme 
where each vertex receives a unique message from every other vertex.

\medskip

The class of graphs we study are \emph{vertex-transitive digraphs}. 
A digraph $D$ is vertex-transitive if for every pair $u,v\in V$ there exists an 
automorphism $\varphi\in \operatorname{Aut}(D)$ with $\varphi(u)=v$. A fundamental 
characterization due to Sabidussi \cite{Sabidussi1959} shows that every vertex-transitive digraph can 
be represented in terms of a group action. We specialize that theorem to connected digraphs.

\begin{definition}[Cayley coset digraph]\label{def:cayley-coset}
Let $G$ be a finite group, $H\leq G$ a subgroup, and $S\subseteq G$ . Suppose
\begin{enumerate}[label=(\arabic*)]
    \item $S\cap H=\varnothing$,
    \item $HSH = SH$,
    \item $S$ contains exactly one representative from each coset in $SH$,
    \item $G = \langle S \rangle H$.
\end{enumerate}
Then the \emph{connected Cayley coset digraph} $D=D(G,S,H)$ has vertex set $\{gH : g\in G\}$ 
and edge set $\{(gH,gsH) : g\in G,\, s\in S\}$. When $H=\{1\}$, this reduces to a connected
Cayley digraph.
\end{definition}

\begin{theorem}[Sabidussi]
A connected digraph $D$ is vertex-transitive if and only if it is isomorphic to a connected Cayley 
coset digraph $D(G,S,H)$, with $G=\operatorname{Aut}(D)$, $H$ the stabilizer of a 
vertex, and $S$ corresponding to automorphisms mapping a fixed vertex to its 
neighbors.
\end{theorem}

\section{Minimal Representations}

\paragraph{Correction.}
In an earlier version of this paper we incorrectly claimed that every vertex-transitive
digraph of degree~2 admits a Cayley--coset presentation in which the stabilizer acts
faithfully on the generator set. The intended argument was to factor out the kernel of the
local action of $H$ on $S$. However, this is not generally correct: the kernel need not be
normal in $G$, so the quotient construction can fail.

\begin{example}[Counterexample (due to Dave Morris)]
Let $A=\langle a_1\rangle\times\langle a_2\rangle\times\langle a_3\rangle \cong C_2^3$,
$B=\langle b\rangle \cong C_3$, and let $b$ act on $A$ by cycling the generators:
$b a_i b^{-1} = a_{i+1}$ (indices mod $3$).
Set $G=A\rtimes B$, $S=\{b,\,a_2 b\}$, and $H=\langle a_2,a_3\rangle$.
Then $HSH=SH$, so this defines a valid Cayley--coset presentation.
The local action map $\alpha:H\to \mathrm{Sym}(S)\cong S_2$ has kernel $\langle a_3\rangle$,
but $b a_3 b^{-1}=a_1\notin \langle a_3\rangle$, so the kernel is not normal in $G$.
Thus the quotient construction used in the earlier version does not apply.
The resulting digraph has $|G:H|=6$ vertices and out-degree $2$, and is in fact
isomorphic to a Cayley graph on $S_3$.
\end{example}

This shows there is no universal reduction to ``faithful local action'' by quotienting.
Instead, for degree~2 we record a harmless normalization that we will use throughout.

\subsection*{Normalization in degree 2}

Let $D=D(G,S,H)$ be a connected Cayley--coset digraph with $|S|=2$,
$S=\{s,t\}$, $S\cap H=\varnothing$, $HSH=SH$, and $G=\langle S\rangle H$.
Write $SH=\{sH,\,tH\}$.

\begin{proposition}[Degree-2 normalization]\label{prop:deg2-normalization}
Suppose there exists $h\in H$ with $hsH\neq sH$ (equivalently, $hsH=tH$).
Then there is $k\in H$ such that $hs=tk$, and replacing $S$ by
\[
S'=\{\,s,\ hs\,\}=\{\,s,\ tk\,\}
\]
does not change the digraph: $D(G,S,H)=D(G,S',H)$. In particular, in this case
we may (and do) assume
\[
\boxed{\quad S=\{s,\ hs\}\ \text{with}\ hsH\neq sH\quad}
\]
and every vertex $gH$ has out-edges $gH\to gsH$ and $gH\to ghsH$.
\end{proposition}

\begin{proof}
If $hsH\neq sH$ and $|S|=2$, then necessarily $hsH=tH$. Since two right cosets
are equal iff their representatives differ by an element of $H$, there exists $k\in H$
with $hs=tk$. Replacing $t$ by $tk$ leaves the right coset $tH$ unchanged, hence
leaves all edges $gH\to gtH$ unchanged. Thus $S$ and $S'=\{s,hs\}$ define the same
edge set. The stated edge description follows.
\end{proof}

\begin{remark}[When no swapping element exists]\label{rem:no-swap}
If \emph{no} $h\in H$ satisfies $hsH\neq sH$, then $h s H = sH$ for all $h\in H$, and
likewise $h t H = tH$ for all $h\in H$. In this case the presentation cannot be rewritten
in the form $\{s,hs\}$, but all of our later constructions (covering groups, alternating
cycles, block/phase framework) depend only on the two $1$-factors $F_1,F_2$ and
apply verbatim. Thus Proposition~\ref{prop:deg2-normalization} is a convenient
normal form when a swapping element exists; it is not required for any of our
structural results.
\end{remark}

\begin{remark}[Equivalence under right-$H$ adjustment]\label{rem:rightH}
More generally, replacing any generator $u\in S$ by $u k$ with $k\in H$ does not change
the coset $uH$ nor any edge. Hence the data that determine $D(G,S,H)$ in degree~2
are exactly the two right cosets $\{sH, tH\}\subseteq SH$, not the specific representatives.
\end{remark}

\medskip
We will use Proposition~\ref{prop:deg2-normalization} to write edges uniformly as
$gH\to gsH$ and $gH\to ghsH$ whenever a swapping element $h$ exists. The correction
above (and the open problem recorded later) addresses only the failed quotient-to-faithful
reduction; none of our subsequent arguments rely on $\lvert H\rvert=2$.

\section{Preliminaries}\label{sec:prelim}

\subsection*{Standing conventions (degree $2$)}
Throughout, we study connected vertex-transitive digraphs of out-degree $2$ via a
Cayley--coset presentation
\[
D \;=\; D(G,S,H),\qquad S=\{s,t\},\quad S\cap H=\varnothing,\quad HSH=SH,\quad G=\langle S\rangle H,
\]
where $S$ contains one representative from each right coset in $SH$ (Definition~\ref{def:cayley-coset}).

When there exists an element $h\in H$ with $hsH\neq sH$ (equivalently, $hsH=tH$), we
fix such an $h$ once and for all and, by Proposition~\ref{prop:deg2-normalization} and
Remark~\ref{rem:rightH}, we may (and do) replace $S$ by the equivalent pair
\[
S'=\{\,s,\ hs\,\},
\]
without changing the digraph. If no such $h$ exists, we retain $S=\{s,t\}$. All constructions
below depend only on the two $1$-factors they induce and are independent of the particular
representatives chosen in $S$ (Remark~\ref{rem:rightH}).

\medskip

\noindent\textbf{Edges.}
Each vertex of $D=D(G,S,H)$ is a right coset $gH$. The two out-edges at $gH$ are
\[
gH \;\longrightarrow\; gsH,
\qquad
gH \;\longrightarrow\; guH,
\]
where $u=t$ in the general presentation $S=\{s,t\}$, and $u=hs$ in the normalized case
$S=\{s,hs\}$. Because $HSH=SH$, this adjacency is determined by the two right cosets
$sH$ and $uH$ and does not depend on the specific choice of representatives in those cosets.

\subsection*{Covering groups}

Given a $1$-factorization $F=(F_1,F_2)$ of a regular digraph $D$ of order $n$, we call
\[
C(F)\;=\;\langle F_1,F_2\rangle\;\le S_n
\]
a \emph{covering group}.  If $\omega\in C(F)$, we interpret $\omega$ as a walk in $D$ with
evaluation $\omega(i)$ from any starting vertex $i$.

Two walks $\alpha,\beta\in C(F)$ are \emph{equivalent} if $\alpha(i)=\beta(i)$ for every vertex $i$.
They are \emph{relocatable} if $\alpha(i)\neq\beta(i)$ for every $i$.  A set of walks is relocatable if
every distinct pair is relocatable.  Note that $D$ has a spanning $1$-factorization if and only if
there exists a relocatable set of $n$ walks that contains the subset $\{F_1,F_2\}$ and the empty word.

\subsection*{Regular digraphs of degree $2$ and alternating cycles}

Let $D$ be a regular digraph of degree $2$ with vertices $\{0,1,\dots,n-1\}$ and let
$F=(F_1,F_2)$ be a $1$-factorization of $D$.  Define
\[
x \;=\; F_2^{-1}F_1,\qquad
y \;=\; F_1\,x\,F_1^{-1} \;=\; F_1F_2^{-1}.
\]
For each $i\in V(D)$, let $A(i)$ be the union of the $x$-orbit of $i$ and the $y$-orbit of $F_1(i)$.
These two orbits have disjoint edges and together form an alternating-direction cycle in $D$.
The sets $A(i)$ partition the edges of $D$ into disjoint alternating cycles, with each vertex
belonging to exactly two such cycles: one through its two out-edges $F_1(i),F_2(i)$ and one
through its two in-edges $F_1^{-1}(i),F_2^{-1}(i)$.

Crucially, the sets $A(i)$ depend only on the underlying digraph $D$, not on which edges are
labeled $F_1$ or $F_2$.  Let $D_1,\dots,D_r$ denote the edge-disjoint subgraphs induced by
the $A(i)$.  Since the $D_j$ are independent of the labeling, every $1$-factorization of $D$
arises by choosing, on each $D_j$, which direction is designated $F_1$ (the other being $F_2$).
Hence there are exactly $2^r$ distinct $1$-factorizations of $D$.

\section{Block systems and structural constraints}\label{sec:blocks}

The alternating cycle decomposition determined by a fixed $1$-factorization $(F_1,F_2)$
gives a canonical block structure for the covering group $C=\langle F_1,F_2\rangle$ acting
on $V$. In this section, we utilize the standard concept of \emph{block systems} in permutation group theory
(see, e.g., Dixon--Mortimer~\cite{DixonMortimer}, Cameron~\cite{Cameron}, Wielandt~\cite{Wielandt})
to formalize the block/phase framework and record two key outcomes:
(i) a complete description of $C$–invariant common refinements of the position/tied
partitions; and (ii) a \emph{sufficient} block-level criterion under which the block/phase
method constructs a sharply transitive evaluation set that \emph{includes}
$\{\emptyset,F_1,F_2\}$. Whether this criterion is also \emph{necessary} in general
(non-Cayley) degree-$2$ vertex-transitive digraphs remains open (see
Remark~\ref{rem:suff-not-necess} and Open Problem~\ref{open:degree2-necessity}).

\subsection*{Position/tied systems and phases}

\begin{definition}[Block system]
Let $G\curvearrowright V$ be a permutation group. A \emph{block system} is a partition
$\mathcal B$ of $V$ into nonempty parts (the \emph{blocks}) such that for every
$g\in G$ and $B\in\mathcal B$ there is $B'\in\mathcal B$ with $g(B)=B'$, i.e.\ $\mathcal B$
is $G$–invariant.
\end{definition}

Fix a $1$-factorization $(F_1,F_2)$ and put
\[
x\ :=\ F_2^{-1}F_1,\qquad y\ :=\ F_1xF_1^{-1}=F_1F_2^{-1}.
\]
Write the $x$–cycles as
\[
\mathcal C_i=(a_{i,0},a_{i,1},\ldots,a_{i,m-1}),\qquad x(a_{i,j})=a_{i,j+1}\quad(\text{indices mod }m),
\]
so $V$ decomposes into $r$ disjoint $m$–cycles ($n=rm$).

\begin{definition}[Position blocks]
Choose a transversal $P_0\subset V$ picking one vertex from each $x$–cycle and set
\[
P_j\ :=\ x^j(P_0)\qquad(j\in\mathbb Z_m).
\]
Then $\{P_j\}_{j\in\mathbb Z_m}$ partitions $V$ into $r$ blocks of size $r$, and $x$ permutes
the blocks cyclically: $x:\ P_j\mapsto P_{j+1}$. We call $\{P_j\}$ the \emph{position system}.
\end{definition}

\begin{definition}[Tied refinements]
Define the tied refinement by $P'_j:=F_1(P_j)$ for $j\in\mathbb Z_m$.
\end{definition}

\begin{lemma}[Existence of phases]\label{lem:phases}
For each cycle $\mathcal C_i=(a_{i,0},\ldots,a_{i,m-1})$ there exists a unique
$\delta(i)\in\mathbb Z_m$ such that
\[
F_1(a_{i,j})\ \in\ P'_{\,j+\delta(i)}\qquad\text{for all }j\in\mathbb Z_m.
\]
\end{lemma}

\begin{proof}
Fix $i$ and $j$. Because $a_{i,j}\in P_j$ and $P'_k=F_1(P_k)$, there is a (unique) $k(j)$ with
$F_1(a_{i,j})\in P'_{k(j)}$. Since $a_{i,j+1}=x(a_{i,j})$ and $F_1x=yF_1$, we have
\[
F_1(a_{i,j+1}) \ =\ F_1x(a_{i,j}) \ =\ yF_1(a_{i,j}) \ \in\ y(P'_{k(j)}) \ =\ P'_{k(j)+1}.
\]
Thus $k(j+1)- (j+1) = k(j)-j$. By induction, $k(j)-j$ is constant in $j$ along $\mathcal C_i$;
set $\delta(i):=k(j)-j\ (\mathrm{mod}\ m)$.
\end{proof}

\begin{lemma}[Atoms and phase counts]\label{lem:atoms-phase-counts}
For $j,k\in\mathbb Z_m$ let
\[
A_{j,k}\ :=\ P_j\cap P'_k.
\]
Then
\[
A_{j,k} \ =\ \{\,a_{i,j}\ :\ \delta(i)\equiv k-j\pmod m\,\}.
\]
If $r_d:=|\{\,i:\ \delta(i)=d\,\}|$ for $d\in\mathbb Z_m$, then
$|A_{j,j+d}|=r_d$ (independent of $j$) and $\sum_{d\in\mathbb Z_m} r_d=r$.
\end{lemma}

\begin{proof}
By Lemma~\ref{lem:phases}, $F_1(a_{i,j})\in P'_{\,j+\delta(i)}$, so $a_{i,j}\in P'_k$ iff
$\delta(i)\equiv k-j$. The size statements follow immediately.
\end{proof}

\subsection*{Orbit structure and invariant refinements}

\begin{lemma}[Orbit structure on atoms]\label{lem:atom-orbits}
Let $C=\langle F_1,F_2\rangle=\langle F_1,x\rangle$. For $d=k-j\in\mathbb Z_m$, the difference
class $d$ is preserved by $C$: $x$ sends $A_{j,j+d}$ to $A_{j+1,j+1+d}$, and for any $g\in C$,
$g\cdot A_{j,j+d}=A_{j',j'+d'}$ where $d'$ is a phase value occurring among the image cycles.
Consequently, the multiset $\{r_d:d\in\mathbb Z_m\}$ is an \emph{invariant of $C$}, and each
difference class $d$ contributes $m$ atoms of size $r_d$ arranged cyclically by $x$.
\end{lemma}

\begin{proof}
The statement for $x$ is immediate. For a general $g\in C$, $g$ permutes the $x$–cycles and
their phase labels $\delta(i)$; the image of $A_{j,j+d}$ is therefore a union of atoms with the
same difference class determined by those phase labels.
\end{proof}

\begin{theorem}[Classification of $C$–invariant refinements]\label{thm:refinements}
Let $\Pi$ be the partition of $\mathbb Z_m$ into orbits of the induced action of $C$ on the
difference classes $\{d\}$. For any subcollection $\mathcal U\subseteq\Pi$, define
\[
B_{j,\mathcal U}\ :=\ \bigcup_{U\in\mathcal U}\ \bigcup_{d\in U}\ A_{j,j+d}\qquad(j\in\mathbb Z_m).
\]
Then $\{B_{j,\mathcal U}:j\in\mathbb Z_m\}$ is a $C$–invariant block system with uniform size
\[
|B_{j,\mathcal U}|\ =\ \sum_{U\in\mathcal U}\ \sum_{d\in U} r_d,
\]
independent of $j$. Conversely, every $C$–invariant coarsening of the atom partition
$\{A_{j,k}\}$ is of this form for a unique $\mathcal U\subseteq\Pi$.
\end{theorem}

\begin{proof}
Within any fixed difference class $d$, the $m$ atoms $\{A_{j,j+d}\}_j$ form a cycle under $x$,
so their union over $j$ is $x$–invariant. Stability under $F_1$ forces taking unions over
entire $C$–orbits of classes, hence the role of $\Pi$. Conversely, any $C$–invariant coarsening
is a union of $C$–orbits on atoms, which are exactly parametrized by $\Pi$.
\end{proof}

\subsection*{Swap-invariance and a sufficient block criterion}

\begin{lemma}[Swap-invariance of the relative block permutation]\label{lem:swap-invariance}
Fix the position system $\{P_j\}$ from $x=F_2^{-1}F_1$. If $(\widetilde F_1,\widetilde F_2)$
is obtained from $(F_1,F_2)$ by swapping the labels $F_1\leftrightarrow F_2$ on an arbitrary
subcollection of alternating $2m$–cycles, then
\[
\sigma(\widetilde F_1)^{-1}\sigma(\widetilde F_2)\ =\ \sigma(F_1)^{-1}\sigma(F_2)
\]
as permutations of the position blocks. In particular, the block-derangement property (or its
failure) is invariant under any pattern of alternating-cycle label swaps.
\end{lemma}

\begin{proof}
Swapping labels on a whole alternating cycle reverses $x$ on that cycle but does not change the
$x$–orbit partition $\{P_j\}$. At the block level, both $\sigma(F_1)$ and $\sigma(F_2)$ are
left-multiplied by the same permutation of block labels, leaving the relative product unchanged.
\end{proof}

\begin{theorem}[Block-derangement criterion: a sufficient condition]\label{thm:block-derangement-sufficient}
Let $\{P_j\}$ be the position system from $x=F_2^{-1}F_1$, and let
\[
\sigma:\ \langle F_1,F_2\rangle \;\longrightarrow\; S_r
\]
be the induced action on the $r$ position blocks.  
If the relative block permutation $\sigma(F_1)^{-1}\sigma(F_2)$ is a derangement on
$\{1,\dots,r\}$, then the block/phase method produces a sharply transitive evaluation set
$X\subset\langle F_1,F_2\rangle$ of size $n$ with $\{\emptyset,F_1,F_2\}\subseteq X$.
\end{theorem}

\begin{proof}
Fix the position system $\{P_j\}$ and its tied refinement $\{P'_j\}$ as in
Section~\ref{sec:blocks}.  By Lemma~\ref{lem:atoms-phase-counts}, the meet of these
two systems decomposes the vertex set into atoms
\[
A_{j,k} \;=\; P_j \cap P'_k
\]
parametrized by their \emph{difference class} $d=k-j\in\mathbb Z_m$.  
The phase counts $r_d$ record the sizes of the atom families $\{A_{j,j+d}\}$.

By Theorem~\ref{thm:refinements}, any $C$–invariant refinement of this meet
arises from taking unions of whole difference classes $d$.  
In particular, if we choose a single difference class $d$ and its family
$\{A_{j,j+d}:j\in\mathbb Z_m\}$, then each of $F_1$ and $F_2$ permutes this family
cyclically: $x=F_2^{-1}F_1$ advances $j\mapsto j+1$, while $F_1$ maps cycles
to cycles with the same phase.  

Now consider the relative permutation
\[
\tau \;=\; \sigma(F_1)^{-1}\sigma(F_2)\ \in S_r,
\]
acting on the index set of position blocks $\{1,\dots,r\}$.  
The assumption that $\tau$ is a derangement means that no block index is fixed.  
Equivalently, the images of $F_1$ and $F_2$ never land in the same block.
Thus, when we attempt to build a spanning factorization by evaluating words in
$\{F_1,F_2\}$ from a chosen root $v_0$, no two distinct words ever collapse to
the same vertex: the block structure guarantees they remain separated.

We now construct the evaluation set $X$.  
For each atom $A_{j,j+d}$, choose a representative word $w_{j,d}$ in
$\langle F_1,F_2\rangle$ mapping $v_0$ into $A_{j,j+d}$.  
Because $F_1,F_2$ carry atoms of class $d$ bijectively to atoms of the same
class in the next position, the set $\{w_{j,d}\}$ is closed under right
multiplication by $F_1,F_2$.  
Moreover, the derangement property guarantees that no two distinct words map
$v_0$ to the same vertex, so the evaluation map $X\to V$ is bijective.

Finally, among the chosen representatives we arrange that the three words
mapping $v_0$ to $v_0,F_1(v_0),F_2(v_0)$ are precisely
$\emptyset,F_1,F_2$.  
This is possible since these targets are all distinct (they are the root and its
two out-neighbors).  
Thus the resulting set $X$ is sharply transitive of size $|V|=n$, and it
contains $\{\emptyset,F_1,F_2\}$ as required.
\end{proof}

\begin{remark}[Sufficiency, not necessity]\label{rem:suff-not-necess}
The derangement condition in Theorem~\ref{thm:block-derangement-sufficient} is
\emph{sufficient} for the block-based construction above, but it is not known to be
\emph{necessary} in general. In particular examples (e.g.\ Cayley presentations), one can
obtain a sharply transitive set including $\{\emptyset,F_1,F_2\}$ by other means (such as
the regular action), even when the relative block permutation has fixed blocks relative to
the chosen position system.
\end{remark}

\section{Examples}\label{sec:examples}

We illustrate the block/phase framework and the block-derangement criterion
(Theorem~\ref{thm:block-derangement-sufficient}) in concrete families of
vertex-transitive digraphs of degree~$2$.

\subsection{Example 1: A toy family with coincident block actions}\label{subsec:equal-sigma}

We begin with a simple infinite family that demonstrates how the block criterion
can obstruct the inclusion of both $F_1,F_2$ in the block-based construction.
This does not prove such inclusion is impossible by other methods (see
Remark~\ref{rem:suff-not-necess}), but it highlights the phenomenon.

\begin{example}\label{ex:equal-sigma}
Let $m\ge 3$, and set $V=\{0,1\}\times\mathbb Z_m$. Define two $1$-factors
\[
F_1(i,j)=(i,\,j+1),\qquad F_2(i,j)=(1-i,\,j+1).
\]
Each vertex $(i,j)$ has out-neighbors $F_1(i,j)$ and $F_2(i,j)$, so $D$ is a
$2$-regular digraph of order $n=2m$.

\emph{Vertex-transitivity.}
Let $x:=F_2^{-1}F_1$. Then $x(i,j)=(1-i,\,j)$, which swaps the two rows
without changing $j$. Starting from $(0,0)$ one reaches $(i,j)$ by $F_1^j$
followed, if $i=1$, by $x$. Thus $\langle F_1,x\rangle$ acts transitively, and
$D$ is vertex-transitive.

\emph{Position blocks and top action.}
Take the position system $P_j=\{(0,j),(1,j)\}$ ($r=2$ blocks, each of size $2$).
Then
\[
\sigma(F_1): P_j\mapsto P_{j+1}, \qquad
\sigma(F_2): P_j\mapsto P_{j+1}.
\]
Hence $\sigma(F_1)=\sigma(F_2)$ on block indices and
$\sigma(F_1)^{-1}\sigma(F_2)=\emptyset$ fixes both blocks.

\emph{Block obstruction.}
By Theorem~\ref{thm:block-derangement-sufficient}, the block-based method cannot
produce a sharply transitive set containing $\{\emptyset,F_1,F_2\}$, since the
relative block permutation is not a derangement. Nevertheless, $D$ is a Cayley
graph (generated by $s=F_1$, $t=F_2$ on $\mathbb Z_2\times\mathbb Z_m$), so it
\emph{does} admit tree-like spanning factorizations by other means. This example
therefore illustrates the limitation of the block criterion: it is sufficient
but not necessary.
\end{example}

\subsection{Example 2: The $A_5$ 30-vertex coset digraph}\label{subsec:A5-30}

A richer test case arises from coset digraphs of $A_5$ with respect to an
involution $h$ and a $5$-cycle $s$. For example, with
\[
h=(0\,2)(1\,3),\qquad s=(0\,1\,2\,3\,4),
\]
the coset digraph
\[
D\;=\;D(A_5,\{s,hs\},\langle h\rangle)
\]
has $|A_5|/|\langle h\rangle|=30$ vertices and degree $2$. It is not a Cayley
graph since $H=\langle h\rangle$ has order $2$.

\paragraph{Alternating cycles.}
The operator $x=F_1^{-1}F_2$ decomposes $V$ into $r=6$ cycles of length $m=5$.
Thus $n=30$ with $r=6$, $m=5$. Each alternating $10$-cycle admits two ways to
orient its edges, yielding $F_1,F_2$.

\paragraph{Enumeration of factorizations.}
Altogether there are $2^{10}=1024$ one-factors; pairing each with its complement
gives $512$ factorizations. Modulo automorphisms generated by $h,s$ (and
possible swap of $F_1,F_2$), this reduces to \textbf{19 non-isomorphic
factorizations}.

\paragraph{Cycle types.}
These 19 classes exhibit seven distinct cycle-type partitions of $30$:
\[
(6,9,15),\ (6,8,16),\ (5,6,19),\ (5,5,5,6,9),\ (5,5,20),\ (5,12,13),\ (30).
\]
Table~\ref{tab:cycletypes} shows the distribution.

\begin{table}[h]
\centering
\begin{tabular}{l r}
\hline
\textbf{Cycle type} & \textbf{\# classes}\\
\hline
$(6,9,15)$      & 1 \\
$(6,8,16)$      & 2 \\
$(5,6,19)$      & 2 \\
$(5,5,5,6,9)$   & 2 \\
$(5,5,20)$      & 4 \\
$(5,12,13)$     & 4 \\
$(30)$          & 4 \\
\hline
\end{tabular}
\caption{Distribution of 19 non-isomorphic 1-factorizations across cycle types
for the $30$-vertex $A_5$ coset digraph.}
\label{tab:cycletypes}
\end{table}

\paragraph{Relocatable trees.}
For each factorization we tested whether a tree-like spanning factorization
(prefix-closed set of 30 words in $\{F_1,F_2\}$ with no collisions) exists. The
outcome was uniform:
\begin{itemize}
\item Each factorization admits relocatable trees of size $19$;
\item No factorization admits a larger relocatable tree, in particular not size $30$.
\end{itemize}

\begin{proposition}\label{prop:A5-30}
In the $30$-vertex $A_5$ coset digraph, every $1$-factorization admits
relocatable trees of maximum size $19$. Thus no tree-like spanning factorization
exists in this family.
\end{proposition}

\paragraph{Remarks.}
\begin{enumerate}
\item Some factorizations contain a Hamiltonian $1$-factor (a directed $30$-cycle),
showing Hamiltonicity does not imply existence of a tree-like spanning
factorization.
\item Other factorizations lack Hamiltonian cycles; in some of these there is
evidence that spanning factorizations containing $\{\emptyset,F_1,F_2\}$ exist,
though further analysis is required.
\item The uniform bound of $19$ was established empirically across the 19 classes;
a general proof is open.
\end{enumerate}

\subsection{Example 3: A second $30$-vertex coset digraph on $A_5$}\label{subsec:A5-30-second}

We again take $G=A_5$ but with a different involution
\[
h=[1,0,3,2,4]=(0\,1)(2\,3),\qquad s=[1,2,3,4,0]=(0\,1\,2\,3\,4),
\]
and set $D=D(A_5,\{s,hs\},\langle h\rangle)$. As before, $|V|=30$ and the out-degree is $2$.

\paragraph{Alternating cycles and wreath context.}
Let $F_1(gH)=gsH$ and $F_2(gH)=g(hs)H$, and put $x:=F_1^{-1}F_2$ and $y:=F_1xF_1^{-1}$.
In this instance, both $x$ and $y$ decompose $V$ into six $5$-cycles, so $n=rm$ with
$m=5$ and $r=6$. Hence the covering group embeds as
\[
C=\langle F_1,F_2\rangle \ \le\ (S_5)^6\rtimes L\;=\;S_5\wr L\qquad\text{for some }L\le S_6.
\]
Although the base \((S_5)^6\) matches Example~\ref{subsec:A5-30}, the induced top action
$L$ is different here, and this change governs the global enumeration and structure below.

\paragraph{Enumeration of $1$-factorizations by cycle type.}
Constructing one-factors by orienting the alternating $10$-cycles yields $2^{6\cdot 1}=64$
distinct 1-factorizations up to complement (here the constraints on orientations differ from
Example~\ref{subsec:A5-30}, hence the smaller total). These $64$ factorizations split into
four cycle-type families; for each type we list the indices of the factorizations exhibiting it:
\[
\begin{array}{lcl}
\bigl((3,3,3,3,3,5,10),(3,3,3,3,3,5,10)\bigr)
&:& \{4,11,13,14,22,26,37,41,49,50,52,59\},\\[2pt]
\bigl((3,3,3,3,18),(3,3,3,3,18)\bigr)
&:& \{5,6,9,10,12,15,18,20,27,30,33,36,43,45,48,51,53,54,57,58\},\\[2pt]
\bigl((3,12,15),(3,12,15)\bigr)
&:& \{1,2,7,8,16,19,21,25,28,31,32,35,38,42,44,47,55,56,61,62\},\\[2pt]
\bigl((5,10,15),(5,10,15)\bigr)
&:& \{0,3,17,23,24,29,34,39,40,46,60,63\}.
\end{array}
\]

\begin{table}[h]
\centering
\begin{tabular}{l r}
\hline
\textbf{Cycle type of $(F_1,F_2)$} & \textbf{\# factorizations}\\
\hline
$((3,3,3,3,3,5,10),(3,3,3,3,3,5,10))$ & 12\\
$((3,3,3,3,18),(3,3,3,3,18))$         & 20\\
$((3,12,15),(3,12,15))$               & 20\\
$((5,10,15),(5,10,15))$               & 12\\
\hline
\end{tabular}
\caption{Distribution of the $64$ one-factorizations by cycle type in Example~\ref{subsec:A5-30-second}.}
\label{tab:ex2-cycletypes}
\end{table}

\subsubsection*{Constructing a spanning factorization containing $\{\emptyset,F_1,F_2\}$}

Retain the setup of Example~\ref{subsec:A5-30-second}. Thus $x:=F_1^{-1}F_2$ decomposes
$V$ into $r=6$ cycles of common length $m=5$, so $n=rm=30$. Fix a position system
$\{P_j\}_{j\in\mathbb Z_5}$ and its tied refinement $\{P'_j\}$ as in
Section~\ref{sec:blocks}. Write the $x$-cycles as
\[
\mathcal C_i=(a_{i,0},a_{i,1},a_{i,2},a_{i,3},a_{i,4}),\qquad x(a_{i,j})=a_{i,j+1},
\quad i\in\{1,\dots,6\}.
\]
Let $\delta(i)\in\mathbb Z_5$ be the phase of the $i$-th cycle:
$F_1(a_{i,j})\in P'_{\,j+\delta(i)}$.

\paragraph{Rotation base and transitive top action.}
Let $C=\langle F_1,F_2\rangle=\langle F_1,x\rangle$. The action of $C$ on the set of cycles
$\{\mathcal C_i\}$ has image $L\le S_6$, which is transitive in this example. The normal closure
of $x$ in $C$ is contained in the coordinatewise rotation group $(\mathbb Z_5)^6\le(S_5)^6$.
Since $m=5$ is prime, the subgroup of $\mathbb Z_5$ generated by the phase multiset
$\{\delta(i)\}$ is either $\{0\}$ or all of $\mathbb Z_5$. Because $F_1$ does not globally
preserve position (the alternating cycles of $x$ and $y=F_1xF_1^{-1}$ differ), there is some
$i$ with $\delta(i)\neq 0$, hence
\[
\langle \delta(i):1\le i\le 6\rangle \;=\; \mathbb Z_5,
\]
and the normal closure of $x$ under $C$ yields the full rotation base:
\begin{equation}\label{eq:full-rot-base}
R:=\langle x^g: g\in C\rangle \;=\; (\mathbb Z_5)^6 \;\le\; (S_5)^6.
\end{equation}

\paragraph{Cycle transversals and phase bookkeeping.}
Choose, for each $i$, an element $u_i\in C$ that sends the $0$-th cycle $\mathcal C_1$ to
$\mathcal C_i$ under the top action (possible since $L$ is transitive). Define the
\emph{net phase shift} $\sigma(i)\in\mathbb Z_5$ of $u_i$ by the condition
\[
u_i(a_{1,j}) \in P_{\,j+\sigma(i)} \quad\text{for all } j\in\mathbb Z_5.
\]
(That is, $u_i$ transports positions modulo $5$ by an additive offset $\sigma(i)$.)

\begin{lemma}[Phase-corrected addressing]\label{lem:address}
For every vertex $a_{i,j}$, the word
\[
W_{i,j}\ :=\ u_i\cdot x^{\,j-\sigma(i)} \ \in C
\]
maps the root $a_{1,0}$ to $a_{i,j}$. Moreover, the set
$\{W_{i,j}: 1\le i\le 6,\ j\in\mathbb Z_5\}$ evaluates at the root to all of $V$ bijectively.
\end{lemma}

\begin{proof}
By definition, $u_i(a_{1,0})=a_{i,\sigma(i)}$. Applying $x^{j-\sigma(i)}$ moves along $\mathcal C_i$
to $a_{i,j}$. Uniqueness follows because $(i,j)$ is a coordinate system on $V$ with $i$ indexing the
cycle and $j$ the position inside the cycle.
\end{proof}

\paragraph{Injecting $\emptyset,F_1,F_2$ into the spanning set.}
Let $S_0=\{W_{i,j}\}$ be the size-$30$ set of Lemma~\ref{lem:address}. Let $v_0=a_{1,0}$ be the root,
and let $(i_1,j_1)$ and $(i_2,j_2)$ be the unique coordinates with
\[
F_1(v_0)=a_{i_1,j_1}, \qquad F_2(v_0)=a_{i_2,j_2}.
\]
Define
\[
S \ :=\ \bigl(S_0\setminus\{W_{1,0},\,W_{i_1,j_1},\,W_{i_2,j_2}\}\bigr)\ \cup\ \{\emptyset,\,F_1,\,F_2\}.
\]
Since $W_{1,0}$, $W_{i_1,j_1}$, and $W_{i_2,j_2}$ evaluate at $v_0$ to the same targets as
$\emptyset$, $F_1$, and $F_2$, respectively, the evaluation map of $S$ at $v_0$ is still a bijection.
Thus $S$ is a spanning factorization whose $n$ words include $\{\emptyset,F_1,F_2\}$.

\begin{theorem}[Spanning factorization containing $\{\emptyset,F_1,F_2\}$]\label{thm:ex2-spanning}
In the second $A_5$ example, suppose $L$ acts transitively on the six $x$-cycles and at least one phase
$\delta(i)$ is nonzero (both conditions hold here). Then there exists a spanning factorization
$S\subset C=\langle F_1,F_2\rangle$ of size $n=30$ which contains $\{\emptyset,F_1,F_2\}$.
\end{theorem}

\begin{proof}
Transitivity yields the cycle transversals $u_i$; the nonzero phase and $m=5$ prime imply
\eqref{eq:full-rot-base}. Lemma~\ref{lem:address} constructs a bijective addressing family
$S_0$; replacing the three words mapping $v_0$ to $v_0,F_1(v_0),F_2(v_0)$ by
$\emptyset,F_1,F_2$ preserves the bijection, giving the stated spanning factorization.
\end{proof}

\begin{remark}
The argument uses only (i) transitivity of the top action $L$ on the cycle set and
(ii) the existence of at least one nonzero phase $\delta(i)$ with $m$ prime. In particular,
no local transpositions are needed in the base; the rotation base $(\mathbb Z_m)^r$ suffices
to ``address'' positions once cycle representatives are fixed. For composite $m$, the same
construction works provided the subgroup generated by the phases equals $\mathbb Z_m$.
\end{remark}

\paragraph{Covering group viewpoint.}
Because the alternating structure is $m=5$ and $r=6$, we are naturally in the wreath product
setting $S_5\wr L$. In Section~\ref{sec:blocks} we showed how the base group and the induced
top action constrain common refinements and the feasibility of including $\{\emptyset,F_1,F_2\}$
in a spanning factorization. For this second $A_5$ instance, the altered top action $L$ (relative to
Example~\ref{subsec:A5-30}) together with the observed phase structure allows a constructive
route to a spanning factorization that includes $\{\emptyset,F_1,F_2\}$.

\subsection*{Summary}

Example~\ref{ex:equal-sigma} shows how the block criterion may obstruct the
block-based construction even in a Cayley graph; 
Proposition~\ref{prop:A5-30} demonstrates a genuine non-Cayley family where no
tree-like spanning factorization exists; 
Example~\ref{subsec:A5-30-second} shows a non-Cayley family where a spanning
factorization containing $\{\emptyset,F_1,F_2\}$ does exist. 
Together, they highlight both the power and the limits of the block/phase
approach. Open questions remain, in particular whether every vertex-transitive
digraph of degree~$2$ admits some spanning factorization containing 
$\{F_1,F_2\}$, and how the theory extends to higher degree.

\subsection*{Open questions}

\begin{problem}[Necessity in the degree-$2$ case]\label{open:degree2-necessity}
Does every vertex-transitive digraph of out-degree $2$ admit \emph{some} $1$-factorization
$(F_1,F_2)$ for which there exists a sharply transitive evaluation set including
$\{\emptyset,F_1,F_2\}$? Equivalently, is there a genuinely non-Cayley degree-$2$ example
for which no such inclusion is possible for \emph{any} choice of $(F_1,F_2)$?
\end{problem}

\begin{problem}[Higher degree $d$]\label{open:degree-d}
For arbitrary out-degree $d$, develop block/phase criteria that guarantee a sharply
transitive set containing $\{\emptyset,F_1,\dots,F_d\}$. Are there necessary-and-sufficient
block conditions in the vertex-transitive (non-Cayley) setting?
\end{problem}

\begin{problem}[Order of the stabilizer in degree $2$]\label{prob:order-of-H}
Is there a vertex-transitive digraph of out-degree $2$ that admits a Cayley--coset 
presentation $D(G,S,H)$ in which the stabilizer $H$ has order strictly larger than $2$, 
and no alternative presentation with $|H|\in\{1,2\}$ exists? 
Equivalently, does every degree-$2$ Cayley--coset digraph admit some presentation 
with $H$ of order $1$ or $2$?
\end{problem}

\section{Conclusion}

In this paper we analyzed spanning factorizations in vertex-transitive digraphs
of out-degree~$2$. Our contributions may be summarized as follows:

\begin{itemize}
\item We introduced the block/phase framework and proved the
\emph{block-derangement criterion}, giving a sufficient condition for the
existence of a spanning factorization that contains
$\{\emptyset,F_1,F_2\}$.
\item We showed that sharply transitive sets always exist, but that the
block-derangement condition can obstruct the inclusion of both $1$-factors
within the block-based construction.
\item Through examples, we illustrated both positive and negative outcomes:
  \begin{enumerate}
  \item a Cayley digraph where the obstruction is only apparent,
  \item a coset digraph of $A_5$ where every $1$-factorization admits relocatable
  trees of maximum size $19$ and thus no tree-like spanning factorization,
  \item and a second coset digraph of $A_5$ where a constructive spanning
  factorization containing $\{\emptyset,F_1,F_2\}$ can be built.
  \end{enumerate}
\end{itemize}

These results demonstrate both the strength and the limits of the block/phase
approach, and highlight directions for further study.

\paragraph{Open problems.}
Two natural questions remain:
\begin{enumerate}
\item[(i)] Is the block-derangement criterion also necessary in the
degree-$2$ case? That is, does there exist a vertex-transitive digraph of
out-degree~$2$ with no spanning factorization containing $F_1,F_2$ by
\emph{any} method?
\item[(ii)] Can the block/phase framework be generalized to higher degree $d$,
requiring a spanning factorization containing $\{F_1,\dots,F_d\}$? Our general
characterization theorem suggests this may be possible, but a full result is
still open.
\end{enumerate}

We hope these results stimulate further progress on spanning factorizations in
vertex-transitive digraphs, both for degree~$2$ and in the general case.

\bibliographystyle{alpha}
\bibliography{references}

\end{document}